\documentclass[12pt]{article}
 \usepackage{amsmath}
 
\usepackage{amsfonts}
\usepackage{tikz}
\usepackage[title]{appendix}
\usepackage{subfig}
\usepackage{array}
\usepackage{blkarray}
\usepackage{nth}
\usepackage{graphicx,wrapfig,lipsum}
\usepackage[
  separate-uncertainty = true,
  multi-part-units = repeat
]{siunitx}

\usetikzlibrary{arrows}
\usepackage{epsfig,}
\usepackage{epsfig}
\tikzset{
    vertex/.style = {
        circle,
        draw,
        outer sep = 3pt,
        inner sep = 3pt,
    },edge/.style = {->,> = latex'}
}
\usepackage[affil-it]{authblk}
\usepackage{amssymb}
\usepackage{enumitem}
\usepackage{amsthm}
\usepackage{dsfont}
\def\sp{\mathop{\rm span}}
\def\adj{\mathop{\rm adj}}
\def\Diag{\mathop{\rm Diag}}

\def\rank{\mathop{\rm rank}}
\newcommand{\rr}{\mathbb{R}}
\newcommand{\1}{\mathbf{1}}

\newcommand{\0}{\mathbf{0}}
\newcommand{\pp}{{\mathbf{P}}}

\def\det{{\rm det}}

\newtheorem{theorem}{Theorem}

\newtheorem{corollary}{Corollary}

\newtheorem{lemma}{Lemma}
\newtheorem{definition}{Definition}

\newcommand{\al}{\alpha}
\begin{document}
\begin{center}
\Large{
Moore-Penrose inverse of distance Laplacians of trees are {\bf{Z}} matrices}\\
R. Balaji and Vinayak Gupta 
\end{center}
\begin{center}
  (In memory of Professor Michael Neumann)
  \end{center}
  \begin{center}
 \today
 \end{center}
\footnotetext[1]{Corresponding author: Vinayak Gupta, Email address: vinayakgupta1729v@gmail.com}
\begin{abstract}
We show that all off-diagonal entries in the Moore-Penrose inverse of the 
distance Laplacian matrix of a tree are non-positive.
\end{abstract}
{\bf AMS CLASSIFICATION.} 05C50
\section{Introduction}
Let $T$ be a tree on $n$ vertices $\{1,\dotsc,n\}$. Suppose to each edge $(p,q)$
of $T$, a positive number $w_{pq}$ is assigned. 
We say that $w_{pq}$ is the weight of $(p,q)$.
The distance between any two vertices $i$ and $j$, denoted by $d_{ij}$, 
is the sum of all the weights in the path connecting $i$ and $j$.
The distance matrix of $T$ is then the $n \times n$ matrix with $(i,j)^{\rm th}$ entry equal to
$d_{ij}$ if $i \neq j$ and $d_{ii}=0$ for all $i=1,\dotsc,n$.

Define
\[\eta_{i}=\sum \limits_{j=1}^{n} d_{ij}~~\mbox{and}~~
\Delta:=\Diag(\eta_1,\dotsc,\eta_n).\] The distance Laplacian
of $T$ is then the matrix $L_D:=\Delta-D$. Distance Laplacian matrices are introduced and studied
in \cite{HANSEN}. Recall that the classical Laplacian matrix of $T$ is 
$L:=\nabla -A$, where $A$ is the adjacency matrix of $T$ and $\nabla$ is the diagonal matrix with $\nabla_{ii}$ equal to degree of vertex $i$. We note some properties
common to 
$L_D$ and $L$. 
\begin{enumerate}
\item[\rm (i)] $L_D$ and $L$ are positive semidefinite.
\item[\rm (ii)] Row sums of $L_D$ and $L$ are equal to zero.
\item[\rm (iii)] $\rank(L)=\rank(L_D)=n-1$. 
\item[\rm (iv)] All off-diagonal entries of $L$ and $L_D$ are non-positive.
\end{enumerate}
 In this paper, we deduce a property of distance Laplacian matrices of trees which the classical
 Laplacian matrices do not have. An $n \times n$ real matrix
 $A=[a_{ij}]$ is called a {\bf Z} matrix if all the off diagonal entries of
 $A$ are non-positive. The objective of this paper is to show that
 $L_{D}^{\dag}$ is always a {\bf Z} matrix. In the following example, 
  we see that $L$ is the classical Laplacian of a path on four vertices, but $L^\dag$ is not a {\bf Z} matrix.
 \[L=\left[
\begin{array}{ccccc}  
1  &-1 & 0 & 0\\
   -1  & 2& -1  &0\\
0 & -1  &2 &-1\\
 0&  0& -1&  1\end{array}\right] ~~~~~\mbox{and} ~~~~~L^\dag=
 \left[
\begin{array}{ccccc}
\frac{  7}{8}&   \frac{  1}{8} &  \frac{  -3}{8} & \frac{  -5}{8}\\
\frac{ 1}{8}& \frac{  3}{8}&\frac{ -1}{8}& \frac{ -3}{8}\\ 
 \frac{ - 3}{8}&  \frac{  -1}{8}&  \frac{  3}{8}&  \frac{  1}{8}\\
 \frac{  -5}{8}&  \frac{  -3}{8}&  \frac{  1}{8} & \frac{  7}{8}
\end{array}\right] .
\]

Distance matrices 
are well studied and have many interesting properties and applications: see \cite{GRMAT}. 
Numerical experiments reveal several new results on distance matrices. For example, a perturbation result says that if $D$ is the distance matrix of a tree and $L$ is the Laplacian of some connected graph (with same 
number of vertices), then all entries in $(D^{-1}-L)^{-1}$ are non-negative: 
See Theorem 4.6 in \cite{BKN}. A far reaching generalisation of this result
for matrix weighted trees is shown in
Theorem 2.1 in \cite{BLOCKPER}. Again this result was motivated by numerical experiments.
Our result here is also motivated by numerical computations.

 If $A$ is a {\bf Z}  matrix, then we say it is an {\bf M}  matrix if all the eigenvalues of $A$ have a non-negative real part. 
A question in \cite{PERRON} asks when the group inverse of a singular irreducible  matrix {\bf M} matrix is again an  {\bf M}   matrix.  Another question in resistive electrical networks \cite{STYAN} asks when is
the Moore-Penrose inverse of a Laplacian of a connected graph is an {\bf M} matrix.
Kirkland and Neumann \cite{KIRK} characterized all
weighted trees whose Laplacian have this property. 
 The result in this paper says that
distance Laplacians of trees are irreducible  {\bf M} matrices and their Moore-Penrose inverses are also {\bf M}   matrices. (We note that group inverse and Moore-Penrose inverse coincide for Laplacians and distance Laplacians.)

Our proof techniques in this paper 
works only for trees.
To extend the result in a more general setting, for example, resistance Laplacian matrices of
connected graphs, additional new techniques are certainly required.

\section{Preliminaries}
 \subsection{Notation}
 \begin{enumerate}
 \item[(i)] If $A=[a_{ij}]$ is an $n \times n$ matrix with $(i,j)^{\rm th}$ entry equal to $a_{ij}$, then the matrix $A(i|j)$ will denote the submatrix of $A$ obtained by deleting the $i^{\rm th }$ row and the $j^{\rm th}$ column of $A$. 
 
 \item[(ii)] Let $\Omega_1:=\{s_1,\dotsc,s_k\}$ and $\Omega_2:=\{t_1,\dotsc,t_m\}$ be subsets of
 $\{1,\dotsc,n\}$.  Then $A[\Omega_1,\Omega_2]$ will be the $k \times m$ matrix with $(i,j)^{\rm th}$ entry equal to 
$a_{s_{i}t_{j}}$. So, 
 $A=[a_{s_it
 _j}]$.
 
 \item[(iii)] The vector of all ones in $\rr^n$ will be denoted by $\1$. 
If $m<n$, then $\1_{m}$ will denote the vector of all ones in $\rr^{m}$. 
 The notation $J$ will stand for the symmetric matrix 
 with all entries equal to $1$.

 \item[(iv)] The Moore-Penrose inverse of a matrix $B$ is denoted by $B^\dag$
 and its transpose by $B'$.
 
 \item[(v)] To denote the subgraph induced by a set of vertices $W$, we use the notation $[W]$.   If $a$ and $b$ are any two vertices,
 then $\pp_{ab}$ will be the path connecting $a$ and $b$. The set of all vertices of a subgraph $H$ is denoted by $V(H)$.
 \end{enumerate}
 \subsection{Basic results and techniques}
 We shall use the following results. Let $T$ be a tree on $n$ vertices labelled $\{1,\dotsc,n\}$ where $n \geq 3$. Let
 $D=[d_{ij}]$ be the distance matrix of $T$.
 \begin{enumerate}
 \item[(P1)](Theorem 3.4, \cite{BLOCK})  If $L$ is the Laplacian matrix of a weighted tree and
 $L^\dag=[\alpha_{ij}]$, then
 \[d_{ij}=\alpha_{ii}+\alpha_{jj}-2 \alpha_{ij}. \]
 Because $\rank(L)=n-1$ and $L$ is positive semidefinite, it follows that for any $0\neq x=(x_1,\dotsc,x_n)' \in \rr^n$ such that
 $\sum \limits_{i=1}^{n} x_i=0$, 
 \[\sum_{i,j} x_i x_j d_{ij}=-2 \sum_{i,j} x_i x_j \alpha_{ij} <0. \] 
 
 \item[(P2)] 
  Triangle inequality: If $i,j,k \in \{1,\dotsc,n\}$, then 
 \[d_{ik} \leq d_{ij}+d_{jk}. \] 
 
 \item[(P3)] Let $\nu$ be a positive integer and the sets
$L_1,\dotsc,L_N$ partition $\{1,\dotsc,\nu\}$.
Let $A=[a_{uv}]$ be a $\nu \times \nu$ matrix
such that $A[L_i,L_j]=O$ for all $i<j$. Then 
there exists a permutation matrix $P$
such that
\[P'AP=
\left[
\begin{array}{ccccc}
A[L_1,L_1] & O & \hdots & O\\
 A[L_2,L_1] & A[L_2,L_2] & \hdots & O \\
 \hdots & \hdots & \ddots & \vdots \\
 A[L_N,L_1] & A[L_N,L_2]& \hdots & A[L_N,L_N]\\
\end{array}
\right].
\]
As a consequence,
\begin{enumerate}
\item If $a_{xy}=0$ for all $x \in L_i$, $y \in L_j$ and $i<j$, then $A$ is similar to a block lower triangular matrix with
$i^{\rm th}$ diagonal block equal to $A[L_i,L_i]$.

\item If $a_{xy}=0$ for all $x \in L_i$, $y \in L_j$ and $i>j$, then $A$ is similar to a block upper triangular matrix with
$i^{\rm th}$ diagonal block equal to $A[L_i,L_i]$. 

\item If $a_{xy}=0$ for all $x \in L_i$, $y \in L_j$ and $i\neq j$, then $A$ is similar to a block diagonal matrix with
$i^{\rm th}$ diagonal block equal to $A[L_i,L_i]$.
\end{enumerate}

\item[(P4)] ({\it Matrix determinant lemma}) Let $A$ be a $m\times m$ matrix and $x,y$ be   $m\times 1$ vectors. Then 
\[\det(A+xy')=\det(A)+y'\adj(A)x.\]
 \end{enumerate}

\section{Main result}
Consider a tree $T$ with vertices labelled $\{1,\dotsc,n\}$. 
If $n=2$, then the result is easy to verify.
We assume $n>2$.
Let $D:=[d_{ij}]$ denote the distance matrix of $T$ and
$\eta_i$ be the $i^{\rm th}$ row sum of $D$. Define $\Delta:=\Diag(\eta_1,\dotsc,\eta_n)$. Let
\[S:=\Delta-D. \]  Each row sum of $S$ is zero. So, all the cofactors of $S$ are equal. Let this common cofactor be
$\gamma$.  
Henceforth, we fix the notation $T,D,\Delta$ and $S$.
Our aim is to show that all the off-diagonal entries of $S^\dag=[s^{\dag}_{ij}]$ are non-positive. We shall show that $s^\dag_{12} \leq 0$. 
A similar argument can be repeated to any other off-diagonal entry of $S^\dag$.
In the first step, we show that $s^{\dag}_{12} \leq 0$ if and only if the determinant of a certain matrix constructed from $D$ is non-negative. 
\subsection{Reformulation of the problem}
We shall find a matrix $R$ such that that $s^{\dag}_{12} \leq 0$ if and only if $\det(R) \geq 0$. The following lemma will be useful.
\begin{lemma} \label{S}
Let $S_{*}:=S+ J$. 
The following items hold.
\begin{enumerate}
\item[{\rm (i)}] $S$ is positive semidefinite and $\rank(S)=n-1.$
\item[\rm (ii)] $S_{*}^{-1}=S^\dag+\frac{J}{n^2}$. 
\item[\rm (iii)] $\det(S_{*})=n^2 \gamma$.
\item[\rm (iv)] Let $C=S(1|2)$. Then,
\[s^{\dag}_{12}=\frac{\1_{n-1}'C^{-1}\1_{n-1}}{n^2}. \]
\end{enumerate}
\end{lemma}
\begin{proof}
It follows from the definition of $S$ that, $S\1=\0$.
Let $0 \neq x\in \1^{\perp}$. By (P1), $x'Dx > 0$. So,
\[x'Sx=x'\Delta x-x'Dx >0.\]
If $x\in \sp\{ \1\}$, then $Sx=\0$.
Hence $S$ is positive semidefinite and $S$ has $n-1$ positive eigenvalues. So, $\rank(S)=n-1$. This proves (i).

Since $S\1=\0$ and $\rank(S)=n-1$, 
\[SS^\dag=I-\frac{J}{n}.\]
From a direct verification, 
\[S_{*}^{-1}=S^\dag+\frac{J}{n^2}.\]
This proves (ii).

Writing $S_{*}=S+\1\1'$, by (P4),
\[\det(S_{*})=\det(S)+\1' \adj(S)\1.\]
As $S \1=0$, $\det(S)=0$. Since all the cofactors of $S$ are equal to $\gamma$, \[\adj(S)=\gamma J.\] 
Therefore, $\det(S_{*})=n^2\gamma$.
Item (iii) is proved.

 Each cofactor of $S$ is $\gamma$ and $C=S(1|2)$. So,
\[\det(C)=\det(S(1|2))=-\gamma.\]
In view of item (ii), 
\begin{equation*}
\begin{aligned}
s^{\dag}_{12}&=(S_{*}^{-1})_{12}-\frac{1}{n^2}\\
&=
\frac{1}{\det(S_*)}(\adj S_*)_{12}-\frac{1}{n^2}\\
&=
\frac{-1}{n^2\gamma}(\det(S_{*}(1|2))-\frac{1}{n^2}\\
&=
-\frac{1}{n^2\gamma}(\det(S(1|2)+J(1|2))-\frac{1}{n^2}.
\end{aligned}
\end{equation*}
By the matrix determinant lemma (P4),
\begin{equation*}
\begin{aligned}
s^{\dag}_{12}
 &=\frac{-1}{n^2\gamma}( \det(C)+\det(C) \1_{n-1}'C^{-1}\1_{n-1} )-\frac{1}{n^2}\\
 &=
\frac{-1}{n^2\gamma}(-\gamma-\gamma \1_{n-1}'C^{-1}\1_{n-1} )-\frac{1}{n^2}\\&=
 \frac{\1_{n-1}'C^{-1}\1_{n-1}}{n^2}.
\end{aligned}
\end{equation*}
Item (iv) is proved and the proof is complete.
\end{proof}

By numerical computations, we note that the sign of $s_{12}^\dag$ depends on the sign of the determinant of a certain matrix
defined from $D$. We introduce this matrix now.
\begin{definition}\label{R}
For any $i,j \in \{3,\dotsc,n\}$, define
\[R_{ij}:=
\begin{cases}
-d_{21}+d_{i1}+d_{2j}-d_{ij} & i\neq j\\
-d_{21}+d_{i1}+d_{2i}+   \sum\limits_{k=1}^{n} d_{ik} & i=j.
\end{cases}
\]
Set $R:=[R_{ij}]$. 
\end{definition}
The order of $R$ is $n-2$. 
By triangle inequality,
\[-d_{21}+d_{i1}+d_{2i} \geq 0~~~~i=3,\dotsc,n.\]
Hence, 
\[R_{ii}>0~~~i=3,\dotsc,n. \]

\begin{lemma}\label{s12detr}
$s^{\dag}_{12} \leq 0$   if and only if    $ \det (R)\geq 0$.

\end{lemma}

\begin{proof}
Let $C=S(1|2)$
and $Q$ be the $(n-1)\times (n-1)$ matrix  
$\left[
\begin{array}{rrrrrrrrrrrrr}
1   & \1_{n-2}\\
\0 & I_{n-2}
 \end{array}\right ]$. 
Then,
$Q^{-1}=\left[
\begin{array}{rrrrrrrrrrrrr}
1   & -\1_{n-2}\\
\0 & I_{n-2}
 \end{array}\right ]$.
By a direct computation, 
\[Q'^{-1}CQ^{-1}=\left[
\begin{array}{rrrrrrr}
 s_{21}  & -s_{21}+s_{23} & \dotsc &  ~~~~~~~~~~~ -s_{21}+s_{2n}\\
-s_{21}+s_{31} &  s_{21}-s_{31}-s_{23}+s_{33}     & \dotsc & s_{21}-s_{31}-s_{2n}+s_{3n}  \\
\hdots ~~~~~ &     \hdots  ~~~    & \ddots &    \hdots~~~~~\\
-s_{21}+s_{n1} &  s_{21}-s_{n1}-s_{23}+s_{n3}   & \dotsc  & s_{21}-s_{n1}-s_{2n}+s_{nn} 
 \end{array}\right ].\]
 The entries of $S$ are
\[s_{ij}=
\begin{cases}
-d_{ij} &  {i\neq j}\\
\sum \limits_{k=1}^{n}  d_{ik}      & {i=j}.
\end{cases}
\]
Hence,
\begin{equation}\label{qcqvalue}
Q'^{-1}CQ^{-1}(1|1)=R.
\end{equation}
We note that $$\det(C)=-\gamma < 0~~\mbox{and}~~\det(Q)=1.$$ So, \[\det(Q'^{-1}CQ^{-1})=-\gamma< 0.  \]
By a simple computation,
\begin{equation}\label{1c1}
(QC^{-1}Q')_{11}=\1_{n-1}'C^{-1}\1_{n-1}.
\end{equation}
Using (\ref{qcqvalue}),
\begin{equation}\label{qcq}
\begin{aligned}
(QC^{-1}Q')_{11}& = \frac{1}{\det(Q'^{-1}CQ^{-1})}\det(Q'^{-1}CQ^{-1}(1|1))\\&=                  -\frac{1}{\gamma}\det (R).
\end{aligned}
\end{equation}
By (\ref{1c1}) and (\ref{qcq}),
\[\1_{n-1}'C^{-1}\1_{n-1}=    -\frac{1}{\gamma}\det (R).\]
By item (iv) in Lemma \ref{S}, it now follows that  
$s^{\dag}_{12} \leq 0$ if and only if  $\det (R)\geq 0$.
\end{proof}

\subsection{A property of $R$}
We now proceed to show that $\det(R) \geq 0$.
The following lemma will be useful in the sequel.

\begin{lemma}\label{big}
Let $\al \in V(\pp_{12})$. Suppose there exists a connected component $\widetilde{X}$ 
of $T \smallsetminus (\al)$ not containing $1$ and $2$. Let $u \in V(\widetilde{X})$ be the vertex
adjacent to $\al$. Consider a connected subgraph $X$ of 
$\widetilde{X}$ containing $u$.
 Put $E:=V(X)$. Then,
$R[E,E]$ is a positive semidefinite matrix.
\end{lemma}
\begin{proof}
If $E=\{u\}$, then $R[E,E]=[R_{uu}]$. As $R_{uu}>0$,
the lemma is true in this case.
Let $|E| \geq 2$. 

{\bf Claim 1:}
$R[E,E]$ is symmetric.\\
Let $r,s \in E$. Recall that
 \begin{equation}\label{rs81}
 R_{rs}=-d_{21}+d_{r1}+d_{2s}-d_{rs}.
 \end{equation}
 By our assumption,
 \begin{equation}\label{r11}
d_{r1}=d_{r\al }+d_{\al 1}~~\mbox {and }~ d_{s2}=d_{s\al }+d_{\al 2}.
 \end{equation}
 By (\ref{rs81}) and (\ref{r11}), 
 \begin{equation}\label{rs812}
  R_{rs}=-d_{21}+d_{r\al }+d_{\al 1}+d_{s\al }+d_{\al 2}-d_{rs}.
 \end{equation}
 Again by our assumption, 
  \begin{equation}\label{r12}
d_{r2}=d_{r\al }+d_{\al 2}~~\mbox {and }~ d_{s1}=d_{s\al }+d_{\al 1}.
 \end{equation}
By (\ref{rs812}) and (\ref{r12}), 
\[ R_{rs}=-d_{21}+d_{s1}+d_{r2}-d_{rs}.\]
 
 The right hand side of the above equation is $R_{sr}$.
 The claim is proved.
 
We know that $u\in E$ and is adjacent to $\al$. 
Let $\Omega$ be the set of all non-pendant vertices in $T$.
Since $X$ is connected, and has at least two vertices,
$u$ is adjacent to a vertex in $E$. Hence,
$u \in E \cap \Omega$. So, $E \cap \Omega$ is a non-empty set.

Let $\delta \in E$ be such that 
 \[d_{\delta \al}= \max  \{  d_{x \al}:x\in E\cap \Omega \}.\]
Since $X$ is a tree, there exists a pendant vertex 
adjacent to $\delta$. 
Without loss of generality, let 
$E=\{x_1,\dotsc,x_{t-1},x_{t}\}$ where $x_1=u$, $x_{t-1}=\delta$ and $x_t$ is pendant
vertex adjacent to $x_{t-1}$.\\
{ \bf Claim 2}:    
\begin{equation}\label{ixt0}
R_{ix_{t-1}}=R_{ix_{t}}~~\mbox{for all}~ i \in \{x_1,x_2,\dotsc,x_{t-2}\}.
\end{equation}
By the definition of $R$,
 \begin{equation}\label{ixt1}
 \begin{aligned}
 R_{ix_{t-1}}-R_{ix_{t}}
 &=-d_{21}+d_{i1}+d_{2x_{t-1}}-d_{ix_{t-1}}-(-d_{21}+d_{i1}+d_{2x_{t}}-d_{ix_{t}})\\
 &=(d_{2x_{t-1}} -d_{2x_{t}} )-( d_{ix_{t-1}}- d_{ix_{t}}    ).
 \end{aligned}
 \end{equation}
Since
$x_t$ is a pendant vertex and is  adjacent to $x_{t-1}$,
we have 
\[d_{2x_{t-1}} -d_{2x_{t}}=-d_{x_tx_{t-1}},\] 
and since $i \in \{x_1,\dotsc,x_{t-2}\}$, we have
\[d_{ix_{t-1}}- d_{ix_{t}}=-d_{x_tx_{t-1}}  . \]
By (\ref{ixt1}), we now get
\[ R_{ix_{t-1}}=R_{ix_{t}}.\]
The claim is true.

{ \bf Claim 3}: 
$R_{x_{t-1}x_{t}}=2d_{{\delta \al}}= 2d_{x_{t-1}\al}$.\\
By definition,
\begin{equation}\label{xtxt-12}
\begin{aligned}
R_{x_{t-1}x_{t}}&= -d_{21}+d_{x_{t-1} 1}+d_{2x_{t}}-d_{x_{t-1}x_{t}}.
\end{aligned}
\end{equation}
Vertices $2$ and $x_t$ belong to different components of $T \smallsetminus(\al)$. Also, 
$x_t$ and $x_{t-1}$ are adjacent and $x_t$ is pendant in $X$.
Hence,
\begin{equation}\label{xtxt-13}
d_{2x_t}=d_{2\al}+d_{\al x_t}=d_{2\al}+d_{\al x_{t-1}}+d_{x_{t-1}x_{t}}.
\end{equation}
By (\ref{xtxt-12}) and (\ref{xtxt-13}),
\[R_{x_{t-1}x_{t}}=-d_{21}+d_{x_{t-1} 1}+d_{2\al}+d_{\al x_{t-1}}.\]
As $\al \in V(\pp_{12})$, $d_{21}=d_{2\al}+d_{\al 1}$. Hence
\begin{equation}\label{xtxt-131}
R_{x_{t-1}x_{t}}=-d_{\al 1}+d_{x_{t-1} 1}+d_{\al x_{t-1}}.
\end{equation}
As $1 \not \in V(\tilde{X})$,  
$d_{x_{t-1} 1}=d_{x_{t-1} \al}+d_{\al 1}$.
Hence by (\ref{xtxt-131}),
\begin{equation}\label{xtxt-132}
R_{x_{t-1}x_{t}}=2d_{\al x_{t-1}}.
\end{equation}
This proves the 
claim.

{\bf  Claim 4 }:  
Diagonal entries of $R[E ,E]$ are greater  than or equal to   $2d_{\delta \al}$.\\
Let $r \in E$. 
By definition,
\begin{equation}\label{123}
R_{rr}=-d_{21} +d_{r 1} + d_{2r} +\sum \limits_{k=1}^{n} d_{r k}.
\end{equation}
 Since $r \in E$, $1\not\in E$ and  $2\not\in E$,
\begin{equation}\label{124}
d_{r 1}=d_{r \al} +d_{1 \al}\mbox{ and }  d_{r 2}=d_{r \al} +d_{2 \al} .
\end{equation}
By (\ref{123}) and (\ref{124}),
\begin{equation}\label{125}
\begin{aligned}
R_{rr}=-d_{21}+d_{r \al} +d_{1 \al}+d_{r \al} +d_{2 \al}+\sum_{k=1}^{n} d_{r k}.
\end{aligned}
\end{equation}
 As $d_{21}=d_{2\al}+d_{\al 1}$, (\ref{125}) simplifies to
\begin{equation}\label{126}
R_{rr}=2d_{r \al}+\sum_{k=1}^{n}d_{r k}.
\end{equation}
Case 1: Suppose $r \notin \{\delta,x_t\}$.\\
Then 
\begin{equation} \label{rrr}
R_{rr}=2d_{r \al}+\sum \limits_{k=1}^{n} d_{r k}\geq 2d_{r \al}+d_{r \delta }+d_{r x_t}.
\end{equation}
 By triangle inequality, \[d_{r \al}+d_{r \delta}\geq d_{\delta \al}~~\mbox{and}~~ d_{r \al}+d_{r x_t}\geq d_{x_t \al}.\]
In view of (\ref{rrr}), 
\[ R_{rr} \geq d_{x_t \al}  
+d_{\delta \al}.                 \]
As $x_t$ is adjacent only to $\delta$, 
\[ R_{rr}\geq d_{x_t \al}  
+d_{\delta \al} =d_{{\delta} \al}+ d_{{\delta} x_t} 
+d_{\delta \al}\geq 2d_{\delta \al}.\]
Case 2: If $r=\delta $, then it is immediate from
(\ref{126}).\\
Case 3: Suppose $r=x_t $.

By (\ref{126}),
\[
R_{x_tx_t}\geq 2d_{x_t \al}.
\]
Since $x_t$ is pendant and adjacent to $\delta$, 
\[
d_{x_t \al}=d_{\delta \al}+d_{x_t \delta}=d_{\delta \al}+ d_{{\delta} x_t}   \geq d_{\delta \al}.
\]
By the two previous inequalities, 
\[R_{x_tx_t}\geq 2d_{\delta \al}.\]
The claim is proved.

Define a $t\times t$ matrix 
\[
P:=\left[
\begin{array}{ccccccccccccc}
2d_{\delta \al} & R_{x_1x_2} & \hdots & R_{x_1x_t} \\
R_{x_2x_{1}} & 2 d_{\delta \al} & \hdots & R_{x_2x_t} \\ 
\hdots & \hdots & \ddots & \hdots \\
R_{x_t x_1} & R_{x_{t} x_{2}} & \hdots & 2 d_{\delta \al}
\end{array}
\right].
\]
Because $R[E,E]$ is a symmetric matrix,
$P$ is symmetric.

{ \bf Claim 5}: $P$ is positive semidefinite.\\
We will prove this by using induction on $|E|$. Suppose $|E|$ has only two vertices.
Write $E=\{x_1,x_2\}$, where $u=x_1$. By a simple verification,
\[P=\left[
\begin{array}{rrrrrrrrrrrrr}
2d_{\alpha u}   &2d_{\alpha u}\\2d_{\alpha u} &2d_{\alpha u}
 \end{array}\right ].   \] 
 So, the claim is true in this case.
Suppose the result is true if $|E|<t$. 
Define a $t \times t$ matrix by
 \[Q_1:= \left[
\begin{array}{rrrrrrrrrrrrr}
I_{t-2} &\0 &\0\\
\0' &1 &-1 \\
\0' &0 &1
 \end{array}\right ].   \] 
 We show that $Q_1'PQ_1$ is positive semidefinite.
Claim 2 and Claim 3 imply that
the last two columns of $P$ are equal. 
Hence, by direct computation,  
\begin{equation}
Q_1'PQ_1=
\left[
\begin{array}{rrrrrrrrrrrrr}
P(x_t|x_t) & \0\\
\0' &0
 \end{array}\right ].
 \end{equation}
  Define \[X':=X \smallsetminus (x_t).\]
  Because $x_t$ is pendant, $X'$ is a connected subgraph of $X$ and $u \in V(X')$. Set
  \[E':=V(X')=\{x_1,\dotsc,x_{t-1}\}, ~\mbox{where}~~ x_1=u. \]
 Define
  \[d_{\mu 1}:=  \max \{d_{x\al}: x \in \Omega \cap E'\}.\]
By induction hypothesis,
\[P_1:=
\left[
\begin{array}{ccccc}
2d_{\mu \al} & R_{x_1 x_{2}} & \hdots & R_{x_1x_{t-1}} \\
R_{x_{2} x_{1}} & 2d_{\mu \al} & \hdots & R_{x_{2}x_{t-1}} \\
\hdots & \hdots & \ddots & \hdots \\
R_{ x_{t-1}x_1} & R_{ x_{t-1}x_2} & \hdots & 2 d_{\mu \al}
\end{array}
\right]
\]
is positive semidefinite.
Put
\[P_2:=
\left[
\begin{array}{ccccc}
2d_{\delta \al}-2 d_{\mu \al} & 0 & \hdots & 0 \\
0 & 2d_{\delta \al}-2 d_{\mu \al} & \hdots & 0 \\
\hdots & \hdots & \ddots & \hdots \\
0 & 0 & \hdots & 2d_{\delta \al}-2 d_{\mu \al}
\end{array}
\right].
\]
Then, \[P(x_t|x_{t})=P_1 + P_2. \]
Since $d_{\delta \al}- d_{\mu \al}\geq 0$, $P(x_t|x_t)$ is positive semidefinite
and so is $P$. This proves the claim. \\
 Define 
 \[\Lambda:=\Diag(R_{x_1 x_1}-2d_{\delta \al},\dotsc,R_{x_t x_t}-2 d_{\delta \al}).\]
Then,
 \[R[E,E]  =P+\Lambda.\]
  By Claim 4, $\Lambda$ is positive semidefinite and by the previous claim, $P$
is positive semidefinite. So, $R[E,E]$ is positive semidefinite.
  The proof is complete.
\end{proof}

\subsection{Partitioning $\{3,\dotsc,n\}$}
Let the degree of vertex $1$ be $m$.
We denote the vertex sets of $m$ components of $T\smallsetminus (1)$ 
by \[V_1',V_2,\dotsc,V_m.\]
Assume $2 \in V_1'$.
 Define 
\[V_1:=V_1'\smallsetminus \{2\}.\]

{\bf Illustration} \\
Consider Figure $1$: $T_{16} \smallsetminus (1) $ has three components.
\begin{figure}[!ht]
\centering
\begin{tikzpicture}
 [scale=1.29,auto=center,every node/.style={circle,fill=black!20}] 
\node[vertex] (1) at  (0,0) {$1$};
\node[vertex] (2) at  (-3,0) {$2$};
\node[vertex] (3) at  (-1,0) {$3$}; 
\node[vertex] (4) at  (-2,0) {$4$};  
\node[vertex] (5) at  (-3.8,-0.51) {$5$};  
\node[vertex] (6) at  (-4.6,-01) {$6$};  
\node[vertex] (7) at  (1,0) {$7$}; 

\node[vertex] (8) at  (1.5,-.75) {$8$};
\node[vertex] (9) at  (1.5,.75) {$9$}; 
\node[vertex] (10) at  (0,1) {${10}$};

\node[vertex] (11) at  (-1,1) {${11}$};
\node[vertex] (12) at  (-1,2) {${12}$};
\node[vertex] (13) at  (-1,-1) {${13}$};
\node[vertex] (14) at  (-2,1) {${14}$};
\node[vertex] (15) at  (-3.8,.51) {${15}$};
\node[vertex] (16) at  (-4.6,1) {${16}$};
 
\draw  (1) to (10);
\draw  (1) to (3);
\draw  (3) to (4);
\draw  (1) to (7);
\draw  (5) to (6);
\draw (5) to (2);

\draw  (9) to (7);
\draw  (7) to (8);
\draw (4) to (2);
\draw (11) to (3);
\draw (12) to (11);
\draw (13) to (3);
\draw (14) to (4);
\draw (2) to (15);
\draw (16) to (15);
\end{tikzpicture}
\caption{Tree $T_{16}$ on 16 vertices} \label{fig_2}
\end{figure}
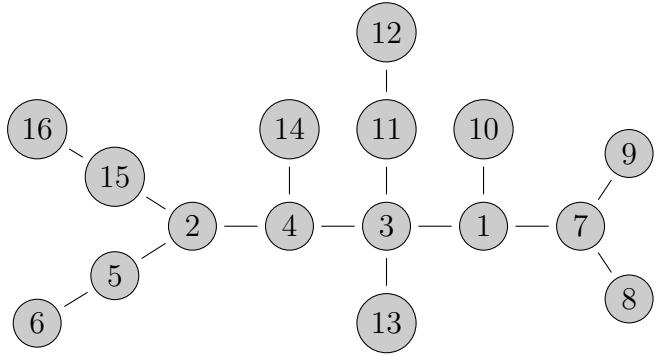
Vertex set of these components are
\[V_1'=\{3,13,11,12,4,14,2,5,6,15,16\},~ V_2=\{10\},~V_3=\{7,9,8\}.\]
Now,
\[ V_1=V_1'\smallsetminus \{2\}=\{3,13,11,12,4,14,5,6,15,16\}.\]

\subsection{A canonical form of $R$}
We now show that $R$ is similar to a block lower triangular matrix. 
\begin{lemma} \label{similar}
$R$ is similar to a block lower triangular matrix with diagonal blocks equal to
$R[V_1,V_1],R[V_2,V_2],\dotsc,R[V_m,V_m]$.
\end{lemma}
\begin{proof}
We know that 
$$V_1 \cup \cdots \cup V_m=\{3,\dotsc,n\}~~\mbox{and}~~V_i \cap V_j=\emptyset. $$ 
Since  
\[R=[R_{\alpha \beta}]~~ ~3 \leq \alpha, \beta \leq n,\] by item (a) in (P3), 
it suffices to show that if $i<j$, $x \in V_i$ and $y \in V_j$, then
 \[ R_{xy}=0.\]
 By definition,
\begin{equation} \label{rxy1}
R_{xy}=-d_{21}+d_{x1}+d_{2y}-d_{xy}.
\end{equation}
Since $x$ and $y$ belong to different components of $T\smallsetminus (1)$,
\begin{equation} \label{dxy1}
d_{xy}=d_{x1}+d_{y1}.
\end{equation}
Using $(\ref{dxy1})$ in (\ref{rxy1}), 
\begin{equation} \label{rxy2}
\begin{aligned}
R_{xy}&=-d_{21}+d_{x1}+d_{2y}-d_{x1}-d_{y1}\\
&=-d_{21}+d_{2y}-d_{y1}.
\end{aligned}
\end{equation} 
We recall that $2\in V_1'$ and $y\in V_j$. Since $i<j$, we see that $1<j$. Hence, $2$ and $y$ belong to different components of $T\smallsetminus (1).$ Thus,
$d_{2y}=d_{21}+d_{y1}$.
By (\ref{rxy2}), 
$R_{xy}=0$.
The proof is complete.
\end{proof}
The following is immediate.
\begin{corollary}\label{detv2}
\[ \det(R)=  \prod_{i=1}^{m}  \det(R[V_i,V_i]).\]
\end{corollary}

\begin{lemma}\label{detvi}
\[\det(R[V_j,V_j])\geq 0 ~~~~~~j=2,\dotsc,m.\]
\end{lemma}
\begin{proof}
Let $j \in \{2,\dotsc,m\}$. 
Put $\widetilde{X}=X=[V_j]$,   $E=
V_j$ and $\al=1$ in Lemma \ref{big}. The result now follows.
\end{proof}

 \subsection{A canonical form of $R[V_1,V_1]$} 
 We partition $V_1$.
Define 
\[V_A:=\{y \in V_1: 2 \notin V(\pp_{1y})\}. \]
\[V_B:=\{y \in V_1: 2 \in V(\pp_{1y})\}.\]
Then,
\[V_{1}=V_{A} \cup V_B~~ \mbox{and}~~ V_A \cap V_B=\emptyset.\]  For example, in Figure \ref{fig_2},
 \[V_A=\{3,4,14,11,13,12\}\mbox{ and } V_B=\{5,6,15,16\}.\]

\begin{lemma}\label{2inpath}
Let $x\in V_B$ and $y\in V_A$. Then $2\in V(\pp_{xy})$.
\end{lemma}
\begin{proof}
Suppose
\begin{equation}\label{xinpxy}
2 \notin V(\pp_{xy}).
\end{equation}
 Since $y \in V_A$,  
\begin{equation} \label{xinpxy2} 
 2 \notin V(\pp_{1y}).
 \end{equation} 
Equations (\ref{xinpxy}) and (\ref{xinpxy2}) imply $2 \notin V(\pp_{1x})$. This contradicts $x \in V_B$. The proof is complete.
\end{proof}

\begin{lemma} \label{vavb}
$R[V_1,V_1]$ similar to a block upper triangular matrix with diagonal blocks equal to
$R[V_A,V_A]$ and $R[V_B,V_B]$.
 \end{lemma}
\begin{proof}
Let $x\in V_B$ and $y\in V_A$. In view of item (b) in (P3), it suffices to show that $R_{xy}=0$. 

By the previous Lemma 
\begin{equation}\label{xy1}
d_{2y}+ d_{2x} =d_{xy}.
\end{equation}
 By Definition \ref{R},
 \begin{equation}\label{xy2}
 R_{xy}=-d_{21}+d_{x1}+d_{2y}-d_{xy}.
 \end{equation}
 As $x\in V_B$,
 \begin{equation}\label{xy3}
  d_{1x}- d_{21} =d_{2x}.
  \end{equation}
 By (\ref{xy2}) and (\ref{xy3}),
 \[R_{xy}=d_{2x}+d_{2y}-d_{xy}.\]
 Equation (\ref{xy1}) now gives
 \[R_{xy}=0.\]
 The proof is complete.
\end{proof} 
 The following is now immediate.
 
 \begin{corollary}\label{detv1}
 \[\det(R[V_1,V_1])=\det(R[V_A,V_A]) \det(R[V_B,V_B]).\]
 \end{corollary}

 \subsection{Partitioning $V_A$}
We partition $V_A$.  
 Let $\pp_{12}$ be the path with vertices 
 $\{1,u_1,\dotsc,u_q,2\}$ and edges
 \[(1,u_1),  (u_1,u_2),\dotsc,(u_q,2).\] 
Define
\[U_i:= \{y\in V_A: d_{u_i y} \leq d_{u_j y}~~~\mbox{for all}~~i\neq j\}~~~i=1,\dotsc,q.\]
 
(We can think  of $U_i$ as the collection of vertices in $V_A$ which are nearer to $u_i$ than $u_j$.)
Clearly 
$u_i \in U_i$.
To illustrate, consider Figure \ref{fig_2}. Here,
\[u_1=3,~U_1=\{3,11,12,13\}, ~u_2=4, ~U_2=\{4,14\}.\]

\begin{lemma}\label{yui}
The following items hold.
\begin{enumerate}
\item[ {\rm (i)}] If $y \in U_i$, then $u_i \in  V(\pp_{y2}) \cap V(\pp_{y1})$.
\item[{\rm (ii)}] If $y \in U_i$, then $u_i \in  V(\pp_{yu_j})$ for $j\neq i$.
\item[{\rm (iii)}] $U_{1},\dotsc,U_q$ partition $V_A$.
\item[{\rm (iv)}]  Let $y \in U_i$ and $z \in U_j$. If $i \neq j$, then 
\[\pp_{yz}=\pp_{yu_i} \cup \pp_{u_i u_{j}} \cup \pp_{u_j z}. \]
\item[{\rm (v)}] Each $[U_i]$ is  a tree. 
\end{enumerate}
\end{lemma}
\begin{proof}
We prove (i) now. 
Assume that $u_i \notin V(\pp_{1y})$. 
Because $u_1$ is the only vertex in $V_1$ adjacent to $1$, $i\neq 1$. 
So, $u_i \notin V(\pp_{u_1y})$.  
Then, $\pp_{u_1y} \cup \pp_{u_1u_i}$ contains $\pp_{yu_i}$.
Now, $u_{i-1} \in V(\pp_{yu_i})$. This implies 
\[d_{u_{i-1}y}<d_{u_iy}.\] But this cannot happen as $y \in U_i$.
So,
\[u_i \in V(\pp_{1y}). \]
Using a similar argument,   \[u_i \in V(\pp_{2y}).\]
 This proves (i).

Let $j > i$. 
 By (i) and from the definition of $u_i$ and $u_{j}$,
    \[u_i \in V(\pp_{2y})~~ \mbox{and}~~
 u_j \in V(\pp_{2u_i}).\]
Thus, 
\[\pp_{2y} =\pp_{2u_{j}} \cup \pp_{u_j u_i} \cup \pp_{u_i y}.\] 
The above equation implies
\[ u_i \in V(\pp_{yu_j})~~\mbox{for all}~j>i. \]
A similar argument leads to
\[ u_i \in V(\pp_{yu_j})~~\mbox{for all}~j<i. \] 
 The proof of (ii) is complete.
 
 Let $y \in U_{i } \cap U_j $, where $j \neq i$. By (ii), it now follows that
 \[u_{i} \in V(\pp_{u_j y}) ~~\mbox{and}~~u_{j} \in V(\pp_{u_i y}).\]
As these two cannot happen simultaneously, $y \notin U_i \cap U_j$. Thus,
\[U_i \cap U_j=\emptyset. \]
By definition of $U_1,\dotsc,U_q$, 
\[U_1 \cup \cdots \cup U_q  \subseteq V_A. \]
Let $x \in V_A$ and $k \in \{1,\dotsc,q\}$ be such that
\[ d_{xu_k}:=\min(d_{xu_1},\dots,d_{x u_q}).\]
Then, $x \in U_k$. Hence
\[V_A \subseteq U_1 \cup \cdots \cup U_q. \]
So,
\[V_A= U_1 \cup \cdots \cup U_q. \]
The proof of (iii) is complete.

The proof of (iv) follows from (ii).

We now show that $[U_i]$ is a tree.
Let $y \in U_i$.  
Since $y,u_i \in V_1'$ and $[V_1']$ is a tree, we have
\begin{equation} \label{vqq}
V(\pp_{yu_i})\subseteq V_1'.
\end{equation}
To show that $[U_i]$ is a tree, it now suffices to show that  $V(\pp_{yu_i})\subseteq U_i$.
Let $x \in V(\pp_{yu_i})$.
Assuming 
  $x \notin U_i$, we shall get a contradiction. By (\ref{vqq}), we now have only three cases:
\[ \mbox{(a)}~ x \in U_j ~\mbox{for some}~ j \neq i~~\mbox{(b)}~x=2~~\mbox{(c)}~x \in V_B.\]
  
  Assume (a).
In view of  item (iv) above, we get 
$u_i \in V(\pp_{xy})$.
But then $x \notin V(\pp_{yu_i})$.
This is a contradiction. So, (a) is not true.

If (b) is true, then $2 \in  V(\pp_{yu_i})$. However (i) implies $u_i \in  V(\pp_{y2})$. This is a contradiction.

If we assume (c), then  $x \in  V(\pp_{yu_i})$. By  Lemma \ref{2inpath}, $2 \in  V(\pp_{yx})$ and therefore
$2 \in  V(\pp_{yu_i})$ implying case (b) is true which is a
contradiction. 

Hence, $V(\pp_{yu_i}) \subseteq U_i$. So, $[U_i]$ is a tree. This completes the proof.
\end{proof}

\subsection{A canonical form of $R[V_A,V_A]$}
We now show that $R[V_A,V_A]$ is similar to a block upper triangular matrix.
\begin{lemma} \label{similarVA}
$R[V_A,V_A]$ is similar to a block upper triangular matrix with $i^{\rm th}$ diagonal block equal to $R[U_i,U_i]$.
\end{lemma}

\begin{proof}
 Let $i>j$. Pick any two elements $r \in U_i$ and $s \in U_j$. 
 By item (c) in (P3), it suffices to show that
 \[R_{rs}=0. \]
We recall that
\begin{equation}\label{Rrs} 
R_{rs}=-d_{21}+d_{r1}+d_{2s}-d_{rs}.
\end{equation}
By item (i) and (iv) of Lemma \ref{yui}, 
\begin{equation} \label{r1}
d_{r1}=d_{ru_i}+d_{u_i 1},~~ d_{2s}=d_{2u_j}+d_{u_j s}~~\mbox{and}~~ d_{rs}= d_{ru_i}+d_{u_ju_i}+d_{su_j}.  
\end{equation}
Thus   (\ref{Rrs}) and (\ref{r1})   give
\begin{align*} 
 R_{rs}&=-d_{21}+d_{ru_i}+d_{u_i 1}+d_{2u_j}+d_{u_j s}-(d_{ru_i}+d_{u_ju_i}+d_{su_j}      )\\&=-d_{21}+d_{u_i 1}+d_{2u_j}-d_{u_ju_i}.
 \end{align*}
Since $i>j$ and $\pp_{12}$ has edges $(u_{k},u_{k+1})$,
\[-d_{21}+d_{u_i1}=-d_{u_i2}~~\mbox{and}~~d_{2u_j} - d_{u_j u_i}=d_{2u_i}. \]
Thus, $R_{rs}=0$.
This completes the proof.
\end{proof}

The following is immediate now.
\begin{corollary}\label{detva}
\[\det(R[V_A,V_A])= \prod_{i=1}^{q} \det(R[{U_i},{U_i}]).\]
\end{corollary}

\subsection{Computation of $\det(R[U_i,U_i])$}
Fix $i \in \{1,\dotsc,q\}$.
We further partition $U_i$ into disjoint sets.
Let $u_i$ have $p_i$ adjacent vertices in $[U_i]$. 
Then, $[U_i] \smallsetminus (u_i)$ have $p_i$ components:
\[G_{i1},\dotsc,G_{ip_i}.\]
Define $Q_{ik}:=V(G_{ik})$.
For example in Figure \ref{fig_2}, for $i=1$, we have
  $$u_1=3,~U_1=\{12,11,3,13\}.$$ There are two components in $[U_1]\smallsetminus (3)$. The vertices of these components are
\[Q_{11}=\{12,11\}~\mbox{and}~ Q_{12}=\{13\}.\]

\begin{lemma} \label{u1u1}
Th following items hold.
\begin{enumerate}
\item[{\rm (i)}] $\det(R[U_i,U_i])=R_{u_i u_i}(\prod \limits_{k =1}^{p_i} \det(R[Q_{ik},Q_{ik}])) $.
\item[{\rm (ii)}] $G_{i1},\dotsc,G_{ip_i}$ are connected components of $T\smallsetminus (u_i)$.
\item[\rm (iii)] \(\det(R[U_i,U_i])\geq 0\).
\end{enumerate}
\end{lemma}
\begin{proof}
Let
$a \in Q_{ir}$, $b \in Q_{is}$ and $r \neq s$.
By definition,
\[R_{ab}=-d_{2 1}+d_{a 1}+d_{2b}-d_{ab}.\]
Since  $u_i \in V(\pp_{12})$,
\begin{equation}\label{zrzs}
 R_{ab}=-d_{2u_i}-d_{u_i1}+d_{a 1}+d_{2b}-d_{ab}.
 \end{equation}
As $a \in U_i$, it follows from item (i) of  Lemma \ref{yui} that 
\begin{equation}\label{zr1}
d_{au_i}=d_{a 1}-d_{1u_i}.
\end{equation}
Using (\ref{zr1}) in (\ref{zrzs}), 
\begin{equation}\label{zrzs2}
 R_{ab}=-d_{2u_i}+d_{u_i a}+d_{2b}-d_{ab}.
 \end{equation}
As $b\in U_i$, it follows from item (i) of Lemma \ref{yui} that 
\begin{equation}\label{zr2}
d_{bu_i}=d_{2b}-d_{2u_i}.
\end{equation}
Using (\ref{zr2}) in (\ref{zrzs2}), 
\begin{equation}\label{zrzs3}
 R_{ab}=d_{bu_i}+d_{u_i a}-d_{ab}.
 \end{equation}
 Finally, since $a$ and $b$ belong to different components of $[U_i]\smallsetminus (u_i)$,
 
 \begin{equation}\label{zr3}
 d_{bu_i}+d_{u_i a}=d_{ab}.
 \end{equation}
 Using  (\ref{zr3}) in (\ref{zrzs3}) 
 \[R_{ab}=0.\] 
We now show that $R_{u_i x}=0$ for any $x \in Q_{is}$. By definition,
 \[R_{u_i x}=-d_{21}+d_{u_i 1}+d_{2x}-d_{u_i x}.\]
 Since  $u_i$ lies  on $\pp_{12}$,
 \begin{equation}\label{zrzs5}
  R_{u_i x}=-d_{u_i 2}+d_{2x}-d_{u_i x}.
  \end{equation}
  As $x\in Q_{is}\subset U_i$, by part (i) of  Lemma \ref{yui},
\begin{equation}\label{zr4}
d_{xu_i}+d_{2u_i}=d_{2x}.
\end{equation}
  By (\ref{zrzs5}) and (\ref{zr4}), 
  \[R_{u_i x}=0.\]
  Similarly, \[R_{xu_i}=0.\]
By item (c) in (P3), we now conclude that $R[U_i,U_i]$ is similar to a block diagonal matrix with diagonal blocks 
  \[R_{u_i u_i},~~R[Q_{ik},Q_{ik}]~~~~~k=1,\dotsc,p_i.\]
  Hence 
\[\det(R[U_i,U_i])=R_{u_i u_i}(\prod_{k =1}^{p_i} \det(R[Q_{ik},Q_{ik}])) .\]
   This completes the proof of (i).
  
By definition $G_{i1},\dotsc,G_{ip_i}$ are connected components of $[U_i] \smallsetminus (u_i)$. So,
each $G_{ik}$ is connected. Suppose $G_{ik}$ is not a connected component of $T \smallsetminus (u_i)$. Then, there 
exists $v \in V(T) \smallsetminus \{u_i\}$ but not in $Q_{ik}$ such that $v$ is adjacent to a vertex $g \in Q_{ik}$. 

Suppose $v \in Q_{ij}$ for some $j \neq k$. But $Q_{ik}$ and $Q_{ij}$ are components of
$[U_i] \smallsetminus (u_i)$ and hence $u_i \in V(\pp_{gv})$. This is not possible. 

Suppose $v \in U_j$ where $j\neq i$. Then, item (iv) in Lemma \ref{yui} implies $u_i \in V(\pp_{gv})$. This is not possible. 

Suppose $v \in V_B$. Then, Lemma \ref{2inpath} gives $2 \in V(\pp_{vg})$.  Again, this is not possible. 

Let $v \in V_2 \cup \cdots \cup V_m$. Since $g\in V_1$, $1\in \pp_{gv}$.
This is a contradiction.
Thus, $G_{ik}$ is a connected component of $T \smallsetminus (u_i)$. The proof of (ii) is complete.

Fix $k \in \{1,\dotsc,p_i\}$. We use Lemma \ref{big}. Set $\tilde{X}=X=G_{ik}$,  $E=Q_{ik}$ and $\al=u_i$. 
By (ii), $X$ is a connected component of $T \smallsetminus (u_i)$. Hence
$\det(R[Q_{ik},Q_{ik}]) \geq 0$. In view of item (i), $\det(R[U_i,U_i]) \geq 0$.
The proof is complete.
\end{proof}
From Corollary \ref{detva} and Lemma  \ref{u1u1}, we get the following.
\begin{lemma}\label{detvageq}
\[\det(R[V_A,V_A])\geq 0.\]
\end{lemma}
\subsection{ Computation of $\det(R[V_B,V_B])$}
Let $[V_B]$ have $s$ components and let the vertex sets of these components be $$W_1,\dotsc,W_s.$$

For example in Figure \ref{fig_2}, $[V_B]$  has two components
\[W_1=\{5,6\},~W_2=\{15,16\}.\]
\begin{lemma}\label{wcomp}
Let $z_i \in W_i$  and  $z_j \in W_j$ where $i \neq j$.
Then $2 \in V(\pp_{z_iz_j})$.
\end{lemma}
\begin{proof}

Since $z_i$ and $z_j$ belong to different components of $[V_B]$, there must exist a vertex $x$ such that
\[x \in V_1',~~ x \notin  V_B,~~\mbox{and}~~x \in V(\pp_{z_iz_j}).\] 
If $x=2$, then we are done. Now, assume $x \neq 2$.
Then, $x \in V_A$.
Since $z_i\in V_B$, Lemma \ref{2inpath} implies that $2 \in V(\pp_{z_ix})$. This implies $2 \in V(\pp_{z_iz_j})$. 
The proof is complete.
\end{proof}

\begin{lemma}
$[W_{1}],\dotsc,[W_s]$ are connected components of $T\smallsetminus (2)$.
\end{lemma}
\begin{proof}
Each $[W_j]$ is connected. 
Suppose $[W_j]$ is not a component in $T \smallsetminus (2)$.
Then there exists a vertex $g \in W_j$ adjacent to $v \in V(T \smallsetminus (2)) \smallsetminus W_j$.

Let $v \in W_{k}$, where $k\neq j$. 
Then, by Lemma \ref{wcomp}, $2 \in V(\pp_{vg})$. This is not possible.

Suppose $v \in V_A$. Then, by Lemma \ref{2inpath}, $2 \in V(\pp_{vg})$.
This is a contradiction.

Suppose $v \notin V_1'$. Then, $v \in V_2 \cup \cdots \cup V_m$;
hence $1 \in V(\pp_{vg})$. This is a contradiction.

Thus, $[W_j]$ is a component in $T \smallsetminus (2)$. 
This completes the proof.
\end{proof}

\begin{lemma}\label{detvb}
The following items hold.
\begin{enumerate}
\item[{\rm (i)}] $\det(R[V_B,V_B])=\prod \limits_{\nu=1}^{s} \det(R[{W_\nu},{W_\nu}])$.
\item[{\rm (ii)}] $\det(R[W_i,W_i])\geq 0~~~~i=1,\dotsc,s$.

\item[{\rm (iii)}] $\det(R[V_B,V_B]) \geq 0$.
\end{enumerate}
\end{lemma}
\begin{proof}
The sets $W_1,\dotsc,W_s$ partition $V_B$.
Let $a \in W_i$ and $b \in W_j$. We claim that if $i \neq j$, then $R_{ab}=0$.
By definition
\begin{equation}\label{zizj1}
R_{ab}=-d_{21}+d_{a1}+d_{2b}-d_{ab}.
\end{equation}
By Lemma \ref{wcomp}, $2\in V(\pp_{ab})$. Hence
\begin{equation}\label{zizj2}
d_{ab}=d_{a2}+d_{2 b}.
\end{equation}
By (\ref{zizj1}) and (\ref{zizj2}),
\begin{equation}\label{zizj3}
R_{ab}=-d_{21}+d_{a1}  +d_{2b} - d_{a2}  -  d_{2b}=-d_{21}+d_{a1}  - d_{a2} . 
\end{equation}
Since $a \in V_B$, $2\in V(\pp_{a1})$, 
\begin{equation}\label{zizj4}
d_{a1}=d_{a2}+d_{2 1}.
\end{equation}
By (\ref{zizj3}) and (\ref{zizj4}),
 \[R_{ab}=0.\]
By (P3), $R[V_B,V_B]$ is similar to a block diagonal matrix with diagonal blocks \[R[{W_1},{W_1}],\dotsc,R[{W_s},{W_s}].\] Therefore,
 \[\det(R[V_B,V_B])= \prod_{i=1}^{s} \det(R[{W_i},{W_i}]).\]
This completes the proof of (i).

The proof of (ii) follows by substituting $\widetilde{X}=X=[W_i]$, $E=W_i$ and $\al=2$ in Lemma \ref{big}.

Item (iii) is immediate from (i) and (ii).
\end{proof}
Corollary \ref{detv1}, Lemma    \ref{detvageq}  and  Lemma \ref{detvb}   give the following.
\begin{lemma} \label{detvbgeq}
\[\det(R[V_1,V_1])\geq 0.\]
\end{lemma}

\subsection{Proof of main result}
\begin{theorem} \label{main}
The Moore-Penrose inverse of the distance Laplacian matrix of a tree is a {\bf Z} matrix.
\end{theorem}
\begin{proof}
The proof follows from Lemmas \ref{s12detr}, \ref{detvi} and \ref{detvbgeq}. 
\end{proof}

\section{An example}
We can ask if the result true for any Euclidean distance matrix.
Here is a counter example.
Let \[D=\left[
\begin{array}{ccccccccccccc}0  &1  &4 & 9& 16\\
 1 & 0&  1&  4 & 9\\  4 & 1 & 0  &1  &4\\
 9 & 4  &1 & 0 & 1\\
16 & 9 & 4 & 1 & 0
\end{array}\right].\] Using standard techniques, it can be verified that $D$ is an Euclidean distance matrix.
Let $$\Delta:=\Diag(\eta_1,\eta_2,\eta_3,\eta_4,\eta_5)~~\mbox{and}~~~
S:=\Delta-D.$$ 
Then,
\[S^\dag=\left[
\begin{array}{ccccccccccccc}
 \frac{2}{81}&  \frac{-1}{81}&  \frac{-1}{90}& \frac{ -1}{405}& \frac{  1}{810}\\
\frac{ -1}{81} & \frac {19}{405} & \frac { -1}{45} & \frac{ -4}{405} &\frac { -1}{405}\\
 \frac{-1}{90 } & \frac{-1}{45}& \frac {   1}{15} & \frac {  -1}{45} & \frac {  -1}{90}\\
\frac{-1}{405} & \frac{-4}{405} & \frac{  -1}{45} & \frac{ 19}{405} & \frac{ -1}{81}\\
 \frac{1}{810} & \frac{ -1}{405 } & \frac{ -1}{90 } & \frac{-1} {81}  &\frac{   2}{81}
\end{array}\right].\]
We see that
$s_{15}^{\dag}>0$.

Acknowledgement: We thank Prof. R. B. Bapat for introducing this problem. The  result 
in this paper was observed by
Prof. Michael Neumann.

R. Balaji and Vinayak Gupta \\
Department of Mathematics \\
Indian Institute of Technology -Madras \\
Chennai 600036 \\
India.
\end{document}